\documentclass[12pt,a4paper]{article}

\usepackage{graphicx} 
\usepackage{amsmath,amsthm,amscd,amssymb,eucal,mathrsfs}
\usepackage{amsfonts}
\usepackage{latexsym}
\usepackage{graphicx} 
\usepackage{color}
\usepackage{multicol,multirow}
\usepackage{dsfont}
\usepackage[all]{xy}
\usepackage{tcolorbox}

\usepackage{tikz}

\newcommand{\mathset}[1]{{\left\{#1\right\}}}
\newcommand{\absolute}[1]{\left\lvert#1\right\rvert}
\newcommand{\norm}[1]{\left\|#1\right\|}

\newtheorem{thm}{Theorem}[section]

\newtheorem{remark}[thm]{Remark}
\newtheorem{definition}[thm]{Definition}

\newtheorem{example}[thm]{Example}

\newtheorem{Proposition}[thm]{Proposition}

\newtheorem*{Satz*}{Satz}

\newtheorem{Lemma}[thm]{Lemma}

\newtheorem{proposition}[thm]{Proposition}

\newtheorem{Corollary}[thm]{Corollary}

\newtheorem{Assumption}{Assumption}

\DeclareMathOperator{\abelian}{ab}

\DeclareMathOperator{\PGL}{PGL}
\DeclareMathOperator{\SL}{SL}

\DeclareMathOperator{\Res}{Res}
\DeclareMathOperator{\Jac}{Jac}

\DeclareMathOperator{\Hom}{Hom}

\DeclareMathOperator{\Div}{Div}

\DeclareMathOperator{\Char}{char}
\DeclareMathOperator{\divisor}{div}

\DeclareMathOperator{\rank}{Ran}

\title{Heat Equations and Hearing the Genus on $p$-adic Mumford Curves via Automorphic Forms}

\author{Patrick Erik Bradley}

\date{\today}

\begin{document}

\maketitle

\begin{abstract}
A self-adjoint operator is constructed on the $L^2$-functions on the $K$-rational points $X(K)$ of a Mumford curve $X$ defined over a non-archimedean local field $K$. It generates a Feller semi-group, and the corresponding heat equation describes a Markov process on $X(K)$. Its spectrum is non-positive, contains zero and has finitely many limit points which are the only non-eigenvalues, and correspond to the zeros of a given regular differential $1$-form on $X(K)$. This allows to recover the genus of $X$ from the spectrum. The hyperelliptic case allows in principle an explicit genus extraction.
\end{abstract}

\section{Introduction}


The Riemann theta function answers the task of inverting abelian integrals. This is also known as the Jacobi inversion problem which asks for an explicit inverse of the Abel-Jacobi map
\[
Y^{(g)}\to \Jac(Y)
\]
where $Y$ is a compact Riemann surface of genus $g$, $\Jac(Y)$ is its Jacobian variety, and $Y^{(g)}$ is the $g$-fold symmetric product of $Y$. The map is determined by the natural embedding 
\[
Y\to \Jac(Y)
\]
which is also called Abel-Jacobi map. The Riemann theta function can be used to construct holomorphic functions on the Riemann surface. In this article, the corresponding $p$-adic version of the Riemann theta function developped by Gerritzen and van der Put in \cite{GvP1980} is used in order to obtain holomorphic functions on a Mumford curve $X$ (which can be viewed as a $p$-adic analogon of a Riemann surface), whose divisor is controlled by specifying points on the universal covering of $X$, at least in the case of a hyperelliptic Mumford curve. Such a function $f$ is then used in order to define the kernel function  of an integral operator $\mathcal{H}_f$ on spaces of complex-valued functions on $X$, and then study the corresponding heat equation. The general idea of an inverse problem in $p$-adic algebraic geometry is to extract information about an object in this domain (here, a Mumford curve)  via a heat equation or a heat operator. In the present case, it is the genus of $X$ which turns out to be revealed by the spectrum of $\mathcal{H}_f$. This result looks encouraging for the study of more refined inverse problems about Mumford curves in future work. 
\newline

In order to obtain explicit functions on a Mumford curve, it is helpful to have an explicit uniformisation procedure, and also to be able to compute the Riemann theta function in a sufficiently explicit manner. These tasks are performed in the hyperelliptic case. Namely,
the first task relies case on a uniformisation of hyperelliptic Mumford curves by Gerritzen \cite{Gerritzen1985}, whereas the second task uses methods by van Steen for such curves \cite{vanSteen1984,Steen1989}. These methods contribute towards explicit period calculations, as done e.g.\ in the genus $2$ case by Teitelbaum \cite{Teitelbaum1988} and others. Related to this is the problem of explicitly relating the ramification points of a finite cover
$X\to\mathds{P}^1$ to the induced ramification points in the universal covering space $\Omega$ of the Mumford curve $X$. This problem is still unsolved in general, and in the hyperelliptic case it is conjectured that the 
two sets of ramification points have identical relative positions
\cite[IX.2.5.3]{GvP1980}.
\newline

Hearing the shape of structures via $p$-adic analysis began, to the best knowledge of the author, with reconstructing finite graphs via $p$-adic Laplacians \cite{BL_shapes_p}.
In \cite{brad_thetaDiffusionTateCurve}, the author develops Laplacian operators on a Tate curve $E_q$ via theta functions. The difference to the case of Mumford curves of arbitrary genus considered here, is that in the former case, the multiplicative group structure of $E_q$ can be used in order to produce  $q$-invariant  functions with \emph{varying}  zero $x$ and pole $x^{-1}$ for   $K$-rational points $x$ of $E_q$. Whereas here, this not being possible, the approach is to define an invariant function with a \emph{fixed} set of $K$-rational zeros and poles   on the Mumford curve. So, in the genus $1$ case, it is possible to extract information about the existence of $2$-torsion points from the spectrum of the operator. 
\newline

Heat equations on Mumford curves are already developed in \cite{brad_HeatMumf}. However, these rely only on the skeleton or reduction graph of the curve. In contrast to this, here the geometry of $X$ comes into play via choosing a regular differential $1$-form $\omega$ on the $K$-rational points $X(K)$ for a non-archimedean local field $K$. This $\omega$ defines in a well-known way a measure $\absolute{\omega}$ on $X(K)$.
The idea for obtaining an integral operator is in fact, just like in previous work of the author, a slight generalisation of Z\'{u}\~{n}iga-Galindo's method of constructing a $p$-adic operator on a finite graph \cite{ZunigaNetworks}, but this time using the measure $\absolute{\omega}$ and a compatible holomorphic function $f$ on $X$ with $K$-rational divisor.
\newline

In a private communication, M.\ van der Put observed the interesting fact that the space $X(K)$, viewed as a $1$-dimensional $K$-manifold for the 
uninteresting (totally disconnected) topology already seems to
be sufficient to use for a diffusion process, but what should be done is to establish a diffusion process on $X$ endowed with the much more interesting Grothendieck topology.
\newline

The results of this article can be stated in three theorems, where $V(g)\subset X(K)$ denotes the zero set of a function or differential form $g$ (whose zero sets are always assumed $K$-rational):
\newline

\noindent
{\bf Theorem 1.} 
\emph{Assume that $V(\omega)\subseteq V(f)$.
Then the operator $\mathcal{H}_f$ is self-adjoint on the Hilbert space $L^2(X(K)\setminus V(f),\absolute{\omega})$, which further has an orthogonal $\mathcal{H}_f$-invariant decomposition into a part in which analogues of the Kozyrev wavelets are eigenfunctions with negative eigenvalues, and a finite-dimensional part on which $\mathcal{H}_f$ has a non-positive spectrum containing zero, and $n$ accumulation points with $n$ being the size of $V(\omega)$, and these are the only non-eigenvalues in the spectrum.}
\newline

\noindent
{\bf Theorem 2.} 
\emph{There exists a probability measure $p_t(x,\cdot)$ on the Borel $\sigma$-algebra of $X(K)\setminus V(\omega)$ such that the Cauchy problem for the heat equation 
\[
\frac{\partial}{\partial t}h(x,t)-\mathcal{H}_f h(x,t)=0
\]
with continuous initial condition $h_0(x)$ has a unique $C^1$-solution
of the form
\[
h(x,t)=\int_{X(K)\setminus V(\omega)}
h_0(y)p_t(x,\absolute{\omega(y)})
\]
and $p_t(x,\cdot)$ is the transition function of a Markov process whose
paths are right continuous and have no other discontinuities than jumps.}
\newline

\noindent
{\bf Theorem 3.} 
\emph{Given a regular $1$-form $\omega$ on $X$, it is possible to recover the genus $g(X)$ from the spectrum of $\mathcal{H}_f$.}
\newline

The proof of Theorem 1 uses a Hilbert-Schmidt operator on an infinite graph derived from $\mathcal{H}_f$, and the result that Kozyrev wavelets can be extended in a useful way to Mumford curves.
\newline

The approach to proving Theorem 2 is to show that one obtains a Feller semi-group in a very general setting on a Hausdorff space endowed with a positive Radon measure from a sufficiently well-behaved kernel function. This includes the kernel function used for constructing $\mathcal{H}_f$.
\newline

Theorem 3 is shown by examining the tree of a good fundamental for the action of the Schottky group after removing the zeros of $\omega$.
\newline

The statement of Theorem 3 comes after investigating the case of a hyperelliptic case, in order to obtain an explicit meromorphic function in terms of the ramification divisor of the cover $X\to\mathds{P}^1$ given by the hyperelliptic involution.
This function is to  be used
for  the kernel function of the operator.
It is remarked that in the hyperelliptic case, the genus extraction result of Theorem 3 can in principle be made explicit using this operator.
\newline

The following section constructs a diffusion operator on the $K$-rational points of a Mumford curve from automorphic forms and a holomorphic differential $1$-form. Self-adjointness is verified, aand the spectrum calculated.
Section 3 studies the heat equation for this operator and obtains the Feller semi-group property needed for a Markov process. 
Section 4 puts the framework of $p$-adic Riemann theta functions into a setting needed for constructing kernel functions from these. In particular two forms of their functional equations are formulated and proven, even if, of course, this is implicitly contained in \cite{GvP1980}. Further, the hyperelliptic case is studied in order to write down an explict meromorphic function together with a compatible differential $1$-form from from the equation given by the degree-$2$-cover of the projective line induced by the hyperelliptic involution. In the end, Theorem 3 is proven, and explained that in the hyperelliptic case, the genus extraction can be made explicit in principle.  
\newline

Throughout the article, $K$ will denote a non-archimedean local field with ring of integers $O_K$ having a
prime element $\pi\in O_K$. The Haar measure on $K$ is denoted as $\mu$, but often as $dx$, if $x$ is a variable. It is normalised such that $\mu(O_K)=1$, and the absolute value $\absolute{\cdot}$ on $K$ is chosen such that
\[
\absolute{\pi}=p^{-f}
\]
where $f$ is the degree of the residue field
\[
O_K/\pi O_K
\]
over the finite field $\mathds{F}_p$ with $p$ elements, where $p$ is a prime number. 
The maximal ideal of $O_K$ 
\[
\mathfrak{m}_K=\pi O_K
\]
is also used.
The completion of the algebraic closure of $K$ is written as $\mathds{C}_p$.
Indicator functions will be written as
\[
\Omega(x\in B)
=\begin{cases}
1,&x\in B
\\
0,&x\notin B
\end{cases}
\]
where $B$ is a measurable subset of $K$. Finally, 
a character 
\[
\chi_K\colon K\to S^1
\]
of $K$ is fixed.

\section{Diffusion on a Mumford Curve}

The lecture notes \cite{GvP1980} provide a thorough introduction to Mumford curves. A short overview of these can be found in \cite[\S5.4]{FP2004}.
Let $X=\Omega/\Gamma$ be a Mumford curve of genus
$g=g(X)\ge2$. Assume that its Schottky group $\Gamma$ is faithfully represented as a discrete subgroup of $\PGL_2(K)$, and let
$\Omega(K)\subset\mathds{P}^1(K)$ be the $K$-rational points of the universal covering space $\Omega\subset\mathds{P}^1_K$ for the covering map $\rho\colon\Omega\to X$. 

\subsection{Measures on the $K$-Rational Points}

Let $\omega\in\Omega_{X/K}^1$ be a regular differential $1$-form on the Mumford curve $X$. It is known that $\omega$ has $2g-2$ zeros in $X(\bar{K})$. So, assume that $K$ is sufficiently large to capture these zeros.
Following Weil \cite[Ch.\ II.2.2]{WeilAAG}
or Igusa \cite[Ch.\ 7.4]{IgusaLocalZeta}, obtain a positive Borel measure $\absolute{\omega}$ on the complement of the vanishing locus of $\absolute{\omega}$ in $X(K)$. Since this is locally on each chart $U\to K$ a zero set w.r.t.\ the Haar measure of $K$ \cite[Lem.\ 3.1]{Yasuda2017}, this measure extends to a positive Borel measure $\absolute{\omega}$ on $X(K)$.
In the case of a Tate elliptic curve $E_q$ considered in \cite{brad_thetaDiffusionTateCurve}, this measure is the invariant measure $\absolute{\omega}=\frac{\absolute{dx}}{\absolute{x}}$, where $\absolute{dx}$ is the  Haar measure of $K$ restricted to the annulus which is a fundamental domain for $E_q(K)$. 

\begin{Assumption}\label{assumption}
It is assumed that the differential form $\omega\in\Omega_{X/K}^1$ and the functions $f\in K(X)$ used in the following have their zeros and poles in $K$.
\end{Assumption}

\subsection{Kernel Functions via Automorphic Forms}

A meromorphic function $f\colon\Omega\to K$ is an \emph{automorphic form with constant automorphy factors}, if 
for all $\alpha\in\Gamma$, there exists some $c(\alpha)\in \mathds{C}_p^\times$ such that
\[
f(z)=c(\alpha)\cdot f(\alpha z)
\]
for all $z\in\Omega$.
The map
\[
c\colon\Gamma\to \mathds{C}_p^\times,\;\alpha\mapsto c(\alpha)
\]
 is called the \emph{automorphy factor} of $f$.
 Any automorphy factor $c$ of an automorphic form is
a character:
\[
c\in\Hom(\Gamma,\mathds{C}_p^\times)
\]
whose values are in $\mathds{C}_p^\times$. In fact, the following holds true:

\begin{Lemma}\label{Dual=Automorphy}
The dual group $\Hom(\Gamma,\mathds{C}_p^\times)$ equals the abelian group of all possible automorphy factors of automorphic forms  on $\Omega$ with constant automorphy factors.
\end{Lemma}

\begin{proof}
\cite[Prop.\ VI.3.4]{GvP1980}.
\end{proof}

We will just say ``automorphic form'' instead of ``automorphic form with constant automorphy factors'' in the following.
\newline

A  special kind of automorphic form on $\Omega$ is a function of the form
\[
\theta(a,b;z)=\prod\limits_{\gamma\in\Gamma}\frac{z-\gamma(a)}{z-\gamma(b)}
\]
with $z\in\Omega$.
It is a meromorphic function on $\Omega$ whose zero set is the $\Gamma$-orbit of $a$, and whose pole set the $\Gamma$-orbit of $b$, if $\Gamma a\neq\Gamma  b$, in which case whose multiplicities are all $1$. Otherwise, it has no zeros or poles. Furthermore, it
has the property:
\[
\theta(a,b;z)=c(a,b;\alpha)\cdot\theta(a,b;\alpha(z))
\]
with $c(a,b;\alpha)\in \mathds{C}_p^\times$ 
for any $\alpha\in \Gamma$. In fact, the map 
\[
c(a,b;\cdot)\colon\Gamma\to\mathds{C}_p^\times,\;\alpha\mapsto c(a,b;\alpha)
\]
is an element of $\Hom(\Gamma,\mathds{C}_p^\times)$, i.e.\ it is the automorphy factor of $\theta(a,b;z)$, cf.\ \cite[II.\S2] {GvP1980}.
\newline

\begin{Lemma}\label{automorphicForms}
Any automorphic form $f$ on $\Omega$ can be written as
\[
f(z)=a_0\cdot\theta(a_1,b_1;z)\cdots\theta(a_r,b_r;z)
\]
with $a_0\in\mathds{C}_p$, $a_i,b_i\in \Omega$, $i=1,\dots,r$, and some natural number $r\ge 1$.
\end{Lemma}

\begin{proof}
\cite[Theorem II.3.2]{GvP1980}.
\end{proof}

If $h(z)$ is a $\Gamma$-invariant meromorphic function on $\Omega$, then
\[
\divisor(\rho_*h)\in\Div(X)
\]
 with $\rho\colon\Omega\to X$ the universal covering map,
 is an algebraic divisor
of $X$. This means that only finitely many coefficients in the formal sum of points are non-zero.
\newline

Define
\[
u_\alpha(z)=\theta(a,\alpha a;z)
\]
with $\alpha\in\Gamma$ and $z\in\Omega$.

\begin{Lemma}
If $\alpha\in \Gamma\setminus[\Gamma,\Gamma]$, then $u_\alpha(z)$ is not constant.
\end{Lemma}

\begin{proof}
\cite[Prop.\ II.3.2]{GvP1980}. 
\end{proof}

\begin{remark}
Any automorphy factor is of the form
\[
c(a,b;\alpha)=\frac{u_\alpha(a)}{u_\alpha(b)}
\]
with
$u_\alpha(z)=\theta(a,\alpha a;z)$, cf.\
 \cite[II.2.3.6]{GvP1980}, which is a meromorphic function without zeros or poles.
\end{remark}

Any $\Gamma$-invariant function on $\Omega$ can now be constructed as
follows:

\begin{Proposition}\label{invariantFunctions}
Any $\Gamma$-invariant meromorphic function $h$ on $\Omega$ is of the form
\[
h(z)=\frac{f_1(z)}{f_2(z)}
\]
with $f_1,f_2$ automorphic forms on $\Omega$ whose automorphy factors coincide, and such that
the divisor $\divisor(\rho_*h)$ has degree zero.
\end{Proposition}

\begin{proof}
First observe that
the $\Gamma$-invariant meromorphic functions on $\Omega$ are precisely the ones with trivial automorphy factor. From Lemma \ref{automorphicForms}, observe that the only way to produce an invariant function is to make a product of factors $\theta(a_i,b_i;z)$ resulting in a trivial automorphy factor. This means a grouping of these $\theta$-factors into pairs $\theta(a_i,b_i;z),\theta(a_i',b_i';z)$ 
such that
\[
\forall\alpha\in\Gamma\colon c(a_i,b_i;\alpha)=c(a_i',b_i';\alpha)
\]
for all $i=1,\dots,r$ with some appropriate natural $r\ge1$. Thus,
after continuing the pairing process, if necessary, 
$h$ is of the form
\[
h(z)=\frac{\theta(a_1,b_1;z)\cdots\theta(a_r,b_r;z)}{\theta(a_1',b_1';z)\cdots\theta(a_r',b_r';z)}\cdot u(z)
\]
where 
$u(z)$ is an automorphic form having trivial automorphy factor, i.e.\ a $\Gamma$-invariant meromorphic function on $\Omega$. 
Hence, $\rho_*h(x)$ is an element of the function field $\mathds{C}_p(X)$ of $X$. This means that the degree of the principal divisor $\divisor(\rho_*h)\in\Div(X)$ is zero.
Since 
\[
\divisor(\theta(a,b;z))=\Gamma a-\Gamma b
\]
the assertion now follows.
\end{proof}

Proposition \ref{invariantFunctions} 
thus shows  how to construct meromorphic functions on $X(K)$ by choosing the zeros and poles of the constituing automorphic forms in $\Omega(K)$. In order to construct a kernel function on $X(K)\times X(K)$, some preparation needs to be done.
\newline


\begin{Lemma}\label{limitMeasureZero}
The set $\mathscr{L}\cap K$ is nowhere dense.
\end{Lemma}

\begin{proof}
This is clearly the case.
\end{proof}

Let $f\in K(X)$, viewed as a $\Gamma$-invariant function on $\Omega(K)$ with zero set $V(f)$ and pole set
$P(f)$, and assume that 
\[
V(\omega)\cup V(f)\cup P(f)\subset \Omega(K)
\]
consist of $K$-rational points.
Let $x,y\in F$, where $F\subset K$ is a good fundamental domain for $\Gamma$ acting on $\Omega(K)$.
Since for general $x,y\in F$, the difference $x-y$ can also be a limit point, 
the following nevertheless holds true:

\begin{Lemma}\label{almostEverywhereDefined}
The function
\[
g\colon F\times F\to K,\;(x,y)\mapsto f(x-y)
\]
is defined outside a nowhere dense set.
\end{Lemma}

\begin{proof}
The pole set $P(f)$ is a finite union of $\Gamma$-orbits in $\Omega(K)$. Hence, its 
interior is empty. From
Lemma \ref{limitMeasureZero}, the assertion now follows.
\end{proof}

\begin{Lemma}\label{rationalFunction}
The function 
\[
\absolute{f}\colon F\to\mathds{R},\;x\mapsto\absolute{f(x)}
\]
coincides on $F$ with the absolute value of a rational function on $K$ with $K$-rational zeros and poles in 
$F$.
\end{Lemma}

\begin{proof}
The function $f$ can be written as
\[
f=h\cdot e
\]
where $h$ is a rational function in one variable, and $e$ an analytic function without zeros or poles on $F$  
\cite[Ch.\ II.3.1]
{GvP1980}. Since these are $\Gamma$-invariant, it follows that 
$\absolute{e}$ is constant, and the assertion follows.
\end{proof}

In the following, the theory of Berkovich-analytic spaces is used. There is a commutative diagram:
\begin{align}\label{skeletonDiagram}
\xymatrix{
\Omega\setminus S\,\ar@{^{(}->}[r]\ar[d]_\sigma&\mathds{A}_K^1\setminus S\ar[d]^\sigma
\\
\sigma(\Omega\setminus S)\,\ar@{^{(}->}[r]&I(\mathds{A}_K^1\setminus S)
}
\end{align}
where $\sigma$ is a retraction map onto the skeleton $I(\mathds{A}_K^1\setminus S)$ of the affine analytic space $\mathds{A}_K^1\setminus S$ with
\[
S:=V(\omega)
\subset \Omega(K)
\]
consisting of finitely many $\Gamma$-orbits.
\newline

Obtain from (\ref{skeletonDiagram})
the following inclusion of trees
\begin{align}\label{treeInclusion}
\xymatrix{
\sigma(\Omega(K)\setminus S)\,\ar@{^{(}->}[r]&\mathscr{T}_K
}
\end{align}
where $\mathscr{T}_K$ is the Bruhat-Tits tree of $\PGL_2(K)$. The tree $\sigma(\Omega(K)\setminus S)$ is an infinite tree whose ends are the  set $\mathscr{L}\cup S\subset K$, where $\mathscr{L}$ is the set of limit points of the discrete group $\Gamma\subset\PGL_2(K)$.
\newline

By taking quotients, obtain a metrised graph
\[
G_S:=\sigma(\Omega\setminus S)/\Gamma
\]
and a section
\[
s\colon G_S\to T_S
\]
into a finite metrised tree with $\absolute{S}$ ends attached to it. These two graphs are skeleta of Berkovich-analytic spaces, and the section extends to a section
\[
s\colon X^{an}\to F^{an}\setminus S
\]
as $T_S$ is the skeleton of  $F^{an}\setminus S$.
Now, define the retraction map
\[
\sigma_F\colon F^{an}\setminus S\to T_S
\]
and obtain for
$A\subseteq F(K)\setminus S$ 
an open neighbourhood as follows:
\[
U(A)=\sigma_F^{-1}(\sigma_F(A))\subset F(K)\setminus S
\]
and
\[
U(x):=U(\mathset{x})
\]
for $x\in K$.

\begin{Lemma}\label{almostLocallyConstant}
It holds true that the map
\[
F\to\mathds{R},\;x\mapsto \absolute{g(x,y)}
\]
is constant on each neighbourhood $U(x)$ of $x\in F$, and for each $y\in F$, both outside a  set
of Haar measure zero.
\end{Lemma}

\begin{proof}
First check this for
the map
\[
\tau_z\colon x\mapsto \absolute{x-z}
\]
on $F$ with $z\in O_K$, where this is possible.
Namely, for $x,z\in F\setminus S$, the value
$-\log_{p^f}\absolute{x-z}$ equals the radius of the disc corresponding to the join in $\mathscr{T}_K$ of the vertices $\sigma_F(x)$
and $\sigma_F(z)$, using (\ref{treeInclusion}).
Consequently, the map $\tau_z$ is constant on $U(x)$ for such choices of $z$, and $x\in F\setminus S$. 
By Lemma \ref{rationalFunction},
the function $x\mapsto \absolute{g(x,y)}$ is
a product of factors 
of the form $\tau_z$ with $z=y-a$ for $a\in O_K$, or their multiplicative inverses, and some constant. 
The functions $\tau_z(x), \tau_z(x)^{-1}$ are defined 
on $F\setminus S$ for almost all $z\in O_K$, by Lemma 
\ref{almostEverywhereDefined} and the finiteness of $S\cap F$.
\end{proof}

Define a kernel function
as
\[
H(x,y)=\begin{cases}
g(x,y)=\absolute{f(x-y)},&x-y\notin U(P(f))
\\
1,&x-y\in U(P(f))
\end{cases}
\]
for $x,y\in F\setminus\left(\mathscr{L}\cup S\right)$, viewed as a $\Gamma$-periodic function on $\left(\Omega(K)\setminus S\right)
\times\left(\Omega(K)\setminus S\right)$, or equivalently as a function on 
$\left(X(K)\setminus\bar{S}\right)\times \left(X(K)\setminus \bar{S}\right)$, where $\bar{S}$ is $S$ modulo $\Gamma$.
There are corresponding integral operators defined as follows:
\begin{align}\label{adjacencyOperator}
\mathcal{A}_f\psi(x)&=
\int_{X(K)\setminus\bar{S}}
H(x,y)\psi(y)\absolute{\omega(y)}
\\\label{LaplacianOperator}
\mathcal{H}_f\psi(x)
&=\int_{X(K)\setminus\bar{S}}H(x,y)
(\psi(y)-\psi(x))\absolute{\omega(y)}
\end{align}
for $\psi\in\mathcal{D}(X(K)\setminus\bar{S})$. 
Operator $\mathcal{A}_f$ is called the \emph{adjacency operator}, and $\mathcal{H}_f$ the \emph{Laplacian operator} on $X(K)\setminus \bar{S}$ with kernel function $H_f$.
The corresponding function
\[
\deg_{\mathcal{H}_f}(x)=\int_{X(K)}H(x,y)\absolute{\omega(y)}
\]
is the \emph{degree function} associated with $\mathcal{A}_f$ or
$\mathcal{H}_f$.

\begin{Lemma}\label{niceKernelFunction}
The following statements hold true:
\begin{enumerate}
\item $H_f$ is 
 constant 
 on each set $U(x)\times U(y)$ outside a set of $\absolute{\omega}$-measure zero in $\left(X(K)\setminus\bar{S}\right)^2$.
\item The function $y\mapsto H_f(x,y)$ is bounded on the compact set $F$.
\item The degree function $\deg_{\mathcal{H}_f}(x)$ is defined everywhere on $X(K)\setminus\bar{S}$ and constant on each neighbourhood $U(x)$ of $x\in F\setminus S$, if viewed as a function on $F\setminus S$.
\end{enumerate}
\end{Lemma}

\begin{proof}
1. This follows from Lemma \ref{almostLocallyConstant}, since $\absolute{g(x,y)}=\absolute{g(y,x)}$.
\newline

\noindent
2. From Lemma \ref{rationalFunction}, it follows that it suffices to check boundedness for functions of the form
\[
F\setminus S\to\mathds{R},\;y\mapsto\absolute{x-y-a}^{-1}
\]
for fixed $x,a\in F\setminus S$, whereby it can be assumed that $x-y$ is not a limit point of $\Gamma$. Observe now that the condition $x-y\notin U(P(f))$ means that 
the join in $\mathscr{T}_K$ of $0,\infty$ and the
vertex $\sigma(x-y)$ must be closer to the vertex $v(0,1,\infty)$ in $\mathscr{T}_K$ corresponding to the unit disc than the join of $0,\infty$ and $a$, if $a\neq 0$. 
This now implies in this case that $\absolute{x-y-a}$ cannot be arbitrarily small. If $a=0$, then   this quantity also cannot be arbitrarily small. Hence its multiplicative inverse is bounded for $y\in F\setminus S$.
\newline

\noindent
3. This follows from 2.\ and 1.
\end{proof}

\subsection{Spectrum of the Laplacian Operator} 

The pairing $\langle\cdot,\cdot\rangle_\omega$ on $L^2(X(K)\setminus\bar{S},\absolute{\omega})$ is defined as follows:
\[
\langle\phi,\psi\rangle_{\omega}
=\int_{X(K)\setminus \bar{S}}\phi(x)\overline{\psi(x)}\absolute{\omega(x)}
\]
Notice that a symmetric real-valued kernel function $H(x,y)$ leads to a symmetric integral operator $\mathcal{H}$, since
\begin{align}\nonumber
\langle\phi,\mathcal{H}\psi\rangle_\omega
&=\iint H(x,y)\phi(x)\overline{\psi(y)}\absolute{\omega(y)}\absolute{\omega(x)}
\\\nonumber
&=\iint H(y,x)\phi(x)\overline{\psi(y)}\absolute{\omega(x)}\absolute{\omega(y)}
\\\label{symmetricOperator}
&=\langle\mathcal{H}\phi,\psi\rangle_\omega
\end{align}
for $\phi,\psi\in L^2(X(K)\setminus S)$.

\begin{Lemma}
The operator $\mathcal{H}_f$ has the following properties:
\begin{enumerate}
\item $\mathcal{H}_t$ is a bounded linear operator on $(C(X(K)\setminus\bar{S},\norm{\cdot}_\infty)$.
\item $\mathcal{H}_f$ is a self-adjoint bounded linear operator on $L^2(X(K)\setminus\bar{S},\absolute{\omega})$.
\end{enumerate}
\end{Lemma}

\begin{proof}
1. This follows from
Lemma \ref{niceKernelFunction}.2
\newline

\noindent
2. Since the kernel function $H_f(x,y)$ is symmetric, it follows from (\ref{symmetricOperator})
that $\mathcal{H}_f$ is a symmetric operator on $L^2(X(K)\setminus\bar{S},\absolute{\omega})$ which furthermore is everywhere defined. By the Hellinger-Toeplitz Theorem \cite[Thm.\ 2.10]{Teschl2010}, it follows that $\mathcal{H}_f$ is also bounded on $L^2(X(K)\setminus\bar{S},\absolute{\omega})$.
Hence, the operator is also self-adjoint on that Hilbert space.
\end{proof}

The goal is now to come up with a helpful matrix as in
\cite{brad_thetaDiffusionTateCurve}.
In order do this,  take the setting into a similar situation as in \cite{FS2022}, and construct an
infinite weighted graph $G$ associated with operator $\mathcal{H}_f$ as follows. Take as vertex set $V(G)$ the countable set $\sigma(X(K)\setminus\bar{S})$, and as edge set $E(G)$ all pairs of vertices, i.e.\ the graph is the (non-simple) complete graph on $V(G)=\sigma(X(K)\setminus\bar{S})$. The weight function on $G$ is defined as
\[
m\colon E(G)\to(0,\infty),\;
(v,w)\mapsto\langle\eta_v,\mathcal{A}_f\eta_w\rangle_\omega
\]
where 
\[
\eta_v(x)=\Omega\left(x\in U(v)\right)
\]
is the indicator function of
\[
U(v):=\sigma^{-1}(v)\subset X(K)\setminus\bar{S}
\]
for $v\in V(G)$.
In \cite{FS2022}, only simple graphs were considered. Here, there are no multiple edges either, but loop-edges are also taken into account.
\newline

Observe that $m$ is a symmetric weight function, because $H_f$ is a symmetric function. There is an extension to the vertex set $V(G)$ by setting
\[
m(v)=\sum\limits_{w\in V(G)}m(vw)
\]
for $v\in V(G)$.

\begin{Lemma}\label{summableWeight}
The weight function $m$ is summable, i.e.\
\[
m(G)=\sum\limits_{v\in V(G)}m(v)<\infty
\]
In other words, $m$ defines a finite $\sigma$-additive measure on $G$.
\end{Lemma}

\begin{proof}
It holds true that
\begin{align*}
m(G)&=\sum\limits_{v,w\in V(G)}
\langle\eta_v,\mathcal{A}_f\eta_w\rangle_{\omega}
\\
&=\sum\limits_{v,w\in V(G)}
\int_{U(v)}\int_{U(w)}
H(x,y)\absolute{\omega(x)}\absolute{\omega(y)}<\infty
\end{align*}
where finiteness holds true, because the function
\[
G(x,y)=H(x,y)\absolute{\omega(x)}\absolute{\omega(y)}
\]
has the property that it is bounded on $X(K)\setminus\bar{S}$, and behaves asymptotically as
\[
G(x,y)\sim C\cdot\absolute{x-a}^\alpha\cdot\absolute{y-b}^\beta,\quad x\to a,\;y\to b
\]
with suitable $C>0$, $\alpha,\beta\ge1$, for $a,b\in S\cap F$ which is a finite set.  
\end{proof}

A function $\psi\colon X(K)\setminus \bar{S}\to\mathds{C}$ is
\emph{vertex-constant}, if it is constant on $U(v)$ for all $v\in V(G)$. Define
\[
L^2(X(K)\setminus\bar{S})_\sigma:=\mathset{f\in L^2(X(K)\setminus\bar{S},\absolute{\omega})\mid\text{$f$ is vertex-constant}}
\]
and the Hilbert space
\[
L^2(G,m)=\mathset{f\colon V(G)\to\mathds{C}\mid\langle f,f\rangle_m<\infty}
\]
where
\[
\langle f,g\rangle_m=
\sum\limits_{v\in V(G)}f(v)\overline{g(v)}m(v)
\]
for $f,g$ complex-valued functions on $V(G)$.
One can now relate
$\mathcal{H}_f$ to the Laplacian $\Delta_m$ on $L^2(G,m)$ given by
\[
\Delta_m f(v)=
\sum\limits_{w\in V(G)}m(vw)(f(w)-f(v))
\]
for $v\in V(G)$. The difference between this graph Laplacian and the one from \cite[eq.\ (1)]{FS2022} is the sign and the normalisation of the latter.

\begin{Lemma}\label{matrixElement}
It holds true that
\[
\langle\mathcal{H}_f(\psi\circ\sigma),\eta_v\rangle_\omega=\Delta_m(\sigma_*\psi)(v)
\]
for $\psi\in L^2(X(K)\setminus\bar{S})$ and $v\in V(G)$. \end{Lemma}

\begin{proof}
It holds true that
\begin{align*}
\langle\mathcal{H}_f(\psi\circ\sigma),\eta_v\rangle_\omega&=
\iint_{(X(K)\setminus\bar{S})^2}
H_f(x,y)(\psi(\sigma(y))-\psi(\sigma(x)))\eta_v(x)\absolute{\omega(y)}\absolute{\omega(x)}
\\
&=
\sum\limits_{w\in V(G)}
\int_{U(v)}\int_{U(w)}H(x,y)\psi(\sigma(y))\absolute{\omega(y)}\absolute{\omega(x)}
\\
&-\sum\limits_{w\in V(G)}\int_{U(v)}\int_{U(w)}H(x,y)\absolute{\omega(y)}\psi(\sigma(x))\absolute{\omega(x)}
\\
&=\sum\limits_{w\in V(G)}m(vw)(\sigma_*\psi)(w)
-m(v)(\sigma_*\psi)(v)
\\
&=\sum\limits_{w\in V(G)}
m(vw)\left(\sigma_*\psi(w)-(\sigma_*\psi(v)\right)
\\
&=\Delta_m(\sigma_*\psi)(v)
\end{align*}
which proves the assertion.
\end{proof}

In \cite{Kozyrev2002}, S.V. Kozyrev introduced $p$-adic wavelets and found they 
are  an orthogonal basis  of $L^2(\mathds{Q}_p)$
consisting of  eigenfunctions of the Vladimirov operator \cite[ Thm. 2]{Kozyrev2002}. 
On a non-archimedean local field $K$, they can be written as
\[
\psi_{B,j}(x)=\mu_K(B)^{\frac12}
\chi_K\left(\pi^{1-d}\tau(j)\right)\Omega(x\in B)
\]
where $\tau\colon O_K/\mathfrak{m}_K\to O_K$ is a lift of the residue field of $K$, $j\in \left(O_K/\mathfrak{m}_K\right)^\times$, and $B\subset K$ a disc of radius $\absolute{\pi}^d$. It is well-known that these so-called \emph{Kozyrev wavelet} form an orthonormal basis of $L^2(K)$, and that
\begin{align}\label{zeroMean}
\int_W\psi_{B,j}(y)\absolute{dy}=0
\end{align}
for all measurable sets $W\subset K$ containing $B$ and $j\in O_K\mathfrak{m}_K$, cf.\ \cite[Thm.\ 2.3.9]{XKZ2018}.

\begin{definition}
A \emph{wavelet} on $X(K)$ is a function $\psi\colon X(K)\to\mathds{C}$
supported inside an open $U\subseteq X(K)$ such that there is a local chart $\kappa\colon U\to K$ with 
$\kappa_*\psi$ a Kozyrev wavelet with support inside $U$.
\end{definition}

Immediately see that a wavelet on $X(K)$ is defined by a Kozyrev
wavelet $\psi_{B,j}$ supported inside the fundamental domain $F$. For simplicity, a wavelet on $X(K)$ will
be written as such a  $\psi_{B,j}$ supported in $F$.
\newline

In order to be able to calculate integrals over $X(K)$, it is necessary to look in some detail at differential $1$-forms on $X$. It is known that the space of regular differential $1$-forms on $X$ has dimension $g$.
This is proved with $p$-adic methods in \cite[Ch.\ VI.4]{GvP1980}. In fact, every regular differential $1$-form on $X$ is a linear combination of differential forms of the following type:
\[
\omega_i=\frac{du_i}{u_i}=w_i(z)\,dz,\quad i=1,\dots,g
\]
where
\[
w_i(z)=\frac{u_i'(z)}{u_i(z)},\quad
u_i=\theta(a,\gamma_ia;z)
\]
with $a\in\Omega$, and $\gamma_1,\dots,\gamma_g\in \Gamma$ are a basis for $\Gamma$ compatible with the good fundamental domain $F$ \cite[Prop.\ VI.4.2]{GvP1980}. Since $\omega_i$ is a $\Gamma$-invariant differential form, the function $w_i$ must transform as
\begin{align}\label{weight2Form}
w_i(\alpha z)=(cz+d)^2\,w_i(z)
\end{align}
if $\alpha\in \Gamma$ is represented by the matrix
\[
\begin{pmatrix}
a&b\\c&d
\end{pmatrix}\in\SL_2(K)
\]
as can easily be verified.
In general, any regular differential $1$-form $\omega$ on $X$ is of the form
\[
\omega=w(z)\,dz
\]
where the analytic function $w(z)$ on $\Omega$ satisfies the transformation rule (\ref{weight2Form}).
\newline

A \emph{circle} in $K$ is a subset of the form
\[
S_k(a)=\mathset{x\in K\mid\absolute{x-a}=\absolute{\pi}^{-k}}
\]
where $a$ is a \emph{centre} and $\absolute{\pi}^{-k}$ the \emph{radius} of $S_k(a)$ with $k\in\mathds{Z}$.

\begin{Lemma}\label{waveletMeanZero}
Let $\psi_{B,j}$ be a wavelet on $X(K)$ supported on a disc $B\subset F$. Then
\[
\int_B\psi_{B,j}(y)\absolute{\omega(y)}=
\begin{cases}
0,&B\cap S=\emptyset
\\
I>0,&B\cap S\neq\emptyset
\end{cases}
\]
for all $j\in \left(O_K/\mathfrak{m}_K\right)^\times$. 
\end{Lemma}

\begin{proof}
Write $\omega$ as
\[
\omega=w(x)\,dx
\]
with $w(x)$ a holomorphic function on $\Omega(K)$.
The absolute value
$\absolute{w(x)}$ is the absolute value of a polynomial on $F$, some of whose zeros may be in the holes of $F$. Hence, according to Assumption \ref{assumption}, $\absolute{w}$ can be written as
\[
\absolute{w(x)}
=C\cdot\absolute{x-a_1}\cdots\absolute{x-a_r}
\]
with $a_1,\dots,a_r\in O_K$, and  $C>0$.

\smallskip
Assume first that $B$ does not contain any zeros of $w(x)$. It follows that $B$ must lie inside a circle $S_k(a_i)$
for some sufficiently large $k\in\mathds{N}$ and $i\in\mathset{1,\dots,r}$. Hence,
\begin{align*}
I&=\int_{S_k(a_i)}\psi_{B,j}(y)\absolute{w_i(y)}\absolute{dy}
=C\cdot\absolute{\pi}^k
\int_{S_k(a_i)}\psi_{B,j}(y)\absolute{dy}=0
\end{align*}
where the latter equality holds true, because $B\subset S_k(a_i)$ and by (\ref{zeroMean}). This proves the assertion in the first case.

\smallskip
Now, assume that $B$ contains a zero of $\omega$. Then assume w.l.o.g.\ that $\omega(0)=0$, and assume that the radius of $B$ is $\absolute{\pi}^d$.
Then 
\[
\int_{S_k(0)}\psi_{B,j}(x)\absolute{dx}
=\begin{cases}
 \absolute{\pi}^k(1-\absolute{\pi}),& k\ge d-1
 \\
-\absolute{\pi}^{d-1},&k=d-2
 \\
 0,&k<d-2
 \end{cases}
\]
according to \cite[Lem.\ 3.6]{Rogers2004}. Hence, the oscillatory integral is positive:
\begin{align*}
\int_{B}\psi_{B,j}(x)\absolute{\omega(x)}
&=\sum\limits_{k=d}^\infty\int_{S_k(0)}
\psi_{B,j}(x)\absolute{\omega(x)}
\\
&=
(1-\absolute{\pi})\sum\limits_{k=d}^\infty c_k\absolute{\pi}^k>0
\end{align*}
since 
the values $c_k$ are positive for $k\ge d$.
This now proves the assertion in the remaining case.
\end{proof}

\begin{Lemma}\label{discInVertex}
Every disc $B$ in $F\setminus S$ is contained in some $U(v)$ for some $v\in V(G)$.
\end{Lemma}

\begin{proof}
Let $a\in B$, and let $v=\sigma_F(a)$. Then clearly $B\subset U(a)=U(v)$.
\end{proof}

The significance of Lemma \ref{discInVertex} is that a disc $B$ in $F\setminus S$ corresponds uniquely to a branch of the Bruhat-Tits tree $\mathscr{T}_K$ rooted in a vertex $w$ closest to $T_S$, but not contained in that tree.
If $v\in V(T_S)$ is the vertex closest to $w$, then
the disc $B$ must be contained in $U(v)$. Since each open $U(v)$ can be written as the union of finitely many discs in $F\setminus S$, the possible discs of $F\setminus S$ are determined by the vertices of the tree $T_S$ through Lemma \ref{discInVertex}. 

\begin{Corollary}\label{waveletEigenvalue}
Any wavelet $\psi_{B,j}$ supported in a disc $B\subset F\setminus S$ is an eigenfunction of $\mathcal{H}_f$ with eigenvalue $-\deg_{\mathcal{H}_f}(x)$ for any $x\in B$.
\end{Corollary}

\begin{proof}
This follows from the fact that $\deg_{\mathcal{H}_f}(x)$ is vertex-constant, and from Lemmas \ref{waveletMeanZero} and \ref{discInVertex}. 
\end{proof}

\begin{Proposition}\label{orthogonalDecomposition}
The following statements hold true:
\begin{align*}
\langle\mathcal{H}_f\eta_w,\eta_v\rangle_\omega&=m(vw)-m(v)\delta_{vw}&(1)
\\
\langle\mathcal{H}_f\eta_v,\psi_{B,j}\rangle_\omega
&=0&(2)
\\
\langle\mathcal{H}_f\psi_{B,j},\eta_w\rangle_\omega&=0&(3)
\\
\langle\mathcal{H}_f\psi_{B,j},\psi_{C,k}\rangle&=-\deg_{\mathcal{H}_f}(x)\,\delta_{B,C}&(4)
\end{align*}
for  $j,k\in \left(O_K/\mathfrak{m}_K\right)^\times$, $v,w\in V(G)$,  discs $B,C\subset F\setminus S$, and $x\in B$. 
\end{Proposition}

\begin{proof}
(1). Obtain from Lemma \ref{matrixElement} that
\begin{align*}
\langle\mathcal{H}_f\eta_w,\eta_v\rangle_\omega&=\Delta_m(\delta_w)(v)
\\
&=\sum\limits_{s\in V(G)}
m(vs)(\delta_{w}(s)-\delta_{w}(v))
\\
&=m(vw)-m(v)\delta_{vw}
\end{align*}

\noindent
(2). Observe that
\begin{align*}
\langle\mathcal{H}_f\eta_v,\psi_{B,j}\rangle_\omega
&=
\int_B\int_{U(v)}
H(x,y)(\eta_v(y)-\eta_v(x))\psi_{B,j}(x)\absolute{\omega(y)}
\absolute{\omega(x)}
\\
&=\int_{B\cap U(v)}
H(x,y)\absolute{\omega(y)}\psi_{B,j}(x)\absolute{\omega(x)}
\\
&-\int_{B\cap U(v)}\deg_{\mathcal{H}_f}(x)\psi_{B,j}(x)\absolute{\omega(x)}
\end{align*}
In the case that $B\cap U(v)=\emptyset$, the last two integrals clearly vanish. According to Lemma \ref{discInVertex}, the remaining case is $B\subset U(v)$. Hence, both $H(x,y)$ and $\deg_{\mathcal{H}_f}(x)$ are constant on $B$, and these two integrals both vanish by Lemma \ref{waveletMeanZero}.

\smallskip
\noindent
(3). Observe that
\begin{align*}
\langle\mathcal{H}_f\psi_{B,j},\eta_w\rangle_\omega
&=
\int_{U(w)}\mathcal{H}_f\psi_{B,j}(x)\absolute{\omega(x)}
\\
&=-\int_{U(w)}\deg_{\mathcal{H}_f}(x)\psi_{B,j}(x)\absolute{\omega(x)}
\end{align*}
Again, by Lemma \ref{discInVertex}, the degree function is a constant in $U(w)$, and hence the integral vanishes by Lemma \ref{waveletMeanZero}.

\smallskip\noindent
(4). Using Lemma \ref{waveletEigenvalue}, obtain
\[
\langle\mathcal{H}_f\psi_{B,j},\psi_{C,k}\rangle_\omega=-\deg_{\mathcal{H}_f}(x)\langle\psi_{B,j},\psi_{C,k}\rangle_\omega
\]
for any $x\in B$, because of
 Lemma \ref{discInVertex}. Now, from that same Lemma take that the last inner product value is a constant times
\[
\langle\psi_{B,j},\psi_{C,j}\rangle_{L^2(K)}=0
\]
which is the inner product on $L^2(K)$ w.r.t.\ to the Haar measure on $K$. This value vanishes, because any two distinct Kozyrev wavelets are mutually orthogonal. This proves the last assertion.
\end{proof}

Use the following notation for the operators
\[
\mathscr{H}:=\mathcal{H}_f|_{L^2(X(K)\setminus\bar{S})_\sigma},\quad
\mathscr{A}:=\mathcal{A}_f|_{L^2(X(K)\setminus\bar{S})_\sigma}
\]
acting on the Hilbert space $L^2(X(K)\setminus\bar{S})_\sigma$.
Both of these operators can be identified with Laplacian and adjacency matrices, respectively, for the infinite graph $G$. The operator $\mathscr{H}$ is the version of a helpful matrix as in \cite{brad_thetaDiffusionTateCurve} to be used in the case of a Mumford curve. Its spectrum is of interest, and so the following Lemma addresses first properties of $\mathscr{A}$.

\begin{Lemma}\label{HilbertSchmidtCondition}
The operator $\mathscr{A}$  is a Hilbert-Schmidt operator on $L^2(X(K)\setminus\bar{S})_\sigma$, if and only if $S\subset V(f)$.
In this case, it is a compact operator.
\end{Lemma}

\begin{proof}
First, notice that
\[
\nu(v):=\langle\eta_v,\eta_v\rangle_\omega
=\int_{U(v)}\absolute{\omega}
\sim C_a\absolute{x-a}^\alpha,\quad v\to a\in S
\]
for $v\in V(G)$ with $C_a>0$ and $\alpha\ge 1$.
Further, the functions $\nu(v)^{-\frac12}\eta_v$ with $v\in V(G)$ form an orthonormal basis for $L^2(X(K)\setminus\bar{S})_\sigma$. The Hilbert-Schmidt property means that the squared Hilbert-Schmidt norm
\[
\norm{\mathscr{A}}_{HS}^2=\sum\limits_{v\in V(G)}
\nu(v)^{-1}\langle\mathcal{A}_f\eta_v,\mathcal{A}_f\eta_v\rangle_\omega
\]
is finite. 
Now, observe that
\begin{align*}
\langle\mathcal{A}_f\eta_v,\mathcal{A}_f\eta_v\rangle_\omega
&=
\int_{X(K)\setminus\bar{S}}\int_{U(v)}H_f(x,y)^2\absolute{\omega(y)}\absolute{\omega(x)}
\\
&\sim \tilde{C}_a\absolute{x-a}^\beta,\quad v\to a\in S\cup V(f)
\end{align*}
for $\tilde{C}_a>0$, $\beta\ge 1$.  Furthermore, $\beta>\alpha$ if and only if $a\in S\cap V(f)$.
Hence, $\norm{\mathscr{A}}_{HS}<\infty$, if and only if $S\subset V(f)$, as asserted.
Since any Hilbert-Schmidt operator is compact, the second assertion also follows.
\end{proof}

\begin{thm}\label{spectrumDiffusionOperator}
Assume that $S\subseteq V(f)$.
Then the following statements hold true:
\begin{enumerate}
\item 
There is an orthogonal decomposition
\[
L^2(X(K)\setminus \bar{S},\absolute{\omega})=L^2(X(K)\setminus\bar{S})_\sigma\oplus L^2(X(K)\setminus\bar{S})_0
\]
which is
$\mathcal{H}_f$-invariant. 
\item
The spectrum of $\mathcal{H}_f$ acting on $L^2(X(K)\setminus\bar{S})_0$ consists of the negative degree eigenvalues $-\deg_{\mathcal{H}_f}(x)$ corresponding to wavelets $\psi_{B,j}(x)$ supported in a disc $B\subset F\setminus S$ containing $x$, and with $j\in\left(O_K/\mathfrak{m}_K\right)^\times$.
\item
The spectrum $\Sigma$ of $\mathcal{H}_f$ restricted to $L^2(X(K)\setminus\bar{S})_\sigma$ is non-positive, contains $0$, its infimum in $\mathds{R}$ exists, and has precisely $\absolute{\bar{S}}$  accumulation points, and these are the only points of $\Sigma$ which are not  eigenvalues.
\end{enumerate}
\end{thm}

This is Theorem 1 of the introduction.

\begin{proof}
1. This follows immediately from Proposition \ref{orthogonalDecomposition}.

\smallskip\noindent
2. This is the statement of Corollary \ref{waveletEigenvalue}.

\smallskip\noindent
3. The spectrum of 
\[
\mathscr{H}=\mathscr{A}-\deg(v)I
\]
with
\[
\deg(v):=\deg_{\mathcal{H}_f}(x)
\]
for any $x\in U(v)$,
is non-negative, because it is a Laplacian and $\mathscr{A}$  is a positive matrix. 
The resolvent of $\mathscr{H}$ is the operator
\[
\Res(\lambda;\mathscr{H})=(\mathscr{H}-\lambda I)^{-1}
=\left(\mathscr{A}-(\deg+\lambda I)\right)^{-1}
\]
where $\deg$ means multiplying with the function $\deg(v)$.
Since $\mathscr{A}$ is a Hilbert-Schmidt operator, its spectrum $\Sigma(\mathscr{A})$
is compact, countably infinite, contains $0$ and
consists of eigenvalues $\mu_n$, except for $0$ which is the only limit point \cite[p.\ 103--115]{Rudin1991}. Hence, the spectrum  $\Sigma(\mathscr{H})$ contains elements of the form
\[
\lambda_n=\mu_n-\deg(v)
\]
with $v\in V(G)$, plus their limit points because the spectrum is closed. Since $\deg(v)$ is locally constant on the discrete topological space $V(G)$, it follows that 
these $\lambda_n$ are eigenvalues of $\mathscr{H}$. Further,
$0$ is also an eigenvalue, because constant functions are eigenfunctions of $\mathscr{H}$.
This now means that a limit point of $\Sigma(\mathscr{H})$ is given by an accumulation point of $\deg(v)$. These are obtained by letting vertex $v$ of $T_S$ approach a zero of $f$ which is the same as a zero of $\omega$. This shows that there are $\absolute{\bar{S}}$ limit points in $\Sigma(\mathscr{H})$,
namely those $\lambda_n$ as above for which  $\deg(v)$ tends to zero. And these are not eigenvalues of $\mathscr{H}$. This proves the last assertion.
\end{proof}

Theorem \ref{spectrumDiffusionOperator} generalises in some sense the statement of
\cite[Cor.\ 32]{FS2022}, where random walks on the countably infinite complete graph are studied.

\section{Heat equation on $X(K)$}

\subsection{A very general Feller semigroup}\label{generalFeller}

Let $Z$ be obtained from a compact Hausdorff space $X$ equipped with a positive Radon measure $\nu$ on the Borel $\sigma$-algebra of $X$ by subtracting a subset of Haar measure zero.
Let $\mathscr{L}$ be a bounded Laplacian integral operator on $C(Z)=C(Z,\norm{\cdot}_\infty)$ defined by a non-negative kernel function $k(x,y)$ on a dense subset of $Z\times Z$, and assume that 
\[
\iint_{Z\times Z}k(x,y)\nu(y)\nu(x)<\infty
\]
i.e.\ the integral operator
\[
\mathscr{I}\psi(x)
=\int_Zk(x,y)\psi(y)\nu(y)
\]
is a Hilbert-Schmidt operator. Assume further that
the degree function
\[
\deg(z)=\int_Z k(z,y)\nu(y)
\]
is finite, 
and positive for all $z\in Z$.

\begin{Lemma}\label{generalFellerSemigroup}
The general operator $\mathscr{L}$ on the space $Z$ generates a strongly continuous positive contraction semigroup (aka Feller semigroup) $\mathset{\exp\left(t\mathscr{L}\right)}_{t\ge0}$ on
$C(Z)$.
\end{Lemma}

\begin{proof}
This is shown by checking the requirements for the Hille-Yosida-Ray
Theorem \cite[Ch. 4, Lem. 2.1]{EK1986}, cf.\ 1.-3.\ below.
\newline

1. The domain of $\mathscr{L}$ is dense in $C(Z)$. 
Since $\mathscr{L}$ is bounded, this is
clearly satisfied.
\newline

2. The operator $\mathscr{L}$ satisfies the positive maximum principle. Assume that
\[
\sup\limits_{z\in Z}\psi(z)=\psi(z_0)
\]
for some $z_0\in Z$. Then
\begin{align*}
\mathscr{L}\psi(z_0)&=
\int_Zk(z_0,z)(\psi(z)-\psi(z_0))\nu(z)
\\
&\le \psi(z_0)\int_Z k(z_0,z)(1-1)\nu(z)
\\
&\le 0
\end{align*}
where the first inequality is valid, because the degree function is positive.
\newline

3.
$\rank(\eta I - \mathscr{L})$ is dense in $C(Z)$ for some $\eta >0$.
Follow the method used in the proof of \cite[Lem.\ 4.1]{ZunigaNetworks}. Namely, the condition is equivalent to having some $\eta>0$ such that 
\begin{align}\label{equation}
\left(\eta I-\mathscr{L}\right)u=h
\end{align}
has a solution in $C(Z)$ for any given $h\in C(Z)$. Write (\ref{equation}) as 
\[
u(z)-\int_Z\frac{u(y)\nu(y)}{\eta+\deg(z)}
=\frac{h(z)}{\eta+\deg(z)}
\]
The operator
\[
Tu(z)=\int_Z\frac{u(y)\nu(y)}{\eta-\deg(z)}
\]
acting on $C(Z)$ satisfies
the inequality
\[
\norm{T}\le\frac{\gamma(Z)}{\eta}
\]
with
\[
\gamma(Z)=\sup\limits_{z\in Z}\deg(z)
\]
Now, by taking $\gamma>\gamma(Z)$, obtain an operator $I-T$ which has an inverse on $C(Z)$. This $\eta$ is independent of $h$. 
\newline

Hence, $\mathscr{L}$ generates a Feller semigroup $\mathset{Q_t}_{t\ge0}$. On the other hand, since $\mathscr{L}$ is a bounded linear operator on a Banach space, $\mathset{\exp\left(t\mathscr{L}\right)}$ is a uniformly continuous semigroup whose infinitesimal generator agrees with
that of $\mathset{Q_t}_{t\ge0}$. It follows that $\exp\left(t\mathscr{L}\right)=Q_t$ for $t\ge0$, cf.\ e.g.\ \cite[Theorems 1.2 and 1.3]{Pazy1983}. This proves the assertions. 
\end{proof}

Lemma \ref{generalFellerSemigroup} says that the operator $\mathscr{L}$ defines a heat equation of the form
\begin{align}\label{ZheatEquation}
\left(\frac{\partial}{\partial t}-\mathscr{L}\right)h(t,z)=0
\end{align}
with $t\ge0$, $z\in Z$,
and a corresponding diffusion process on $Z$.

\begin{proposition}\label{generalMarkovProcess}
There is a probability measure $p_t(z,\cdot)$ with $t\ge0$, $z\in Z$, on the Borel $\sigma$-algebra of $Z$ such that the Cauchy problem for  the heat equation (\ref{ZheatEquation}) having initial condition $h(0,z)=h_0(z)\in C(Z)$ has a unique solution in $C^1((0,\infty),Z)$ of the form
\[
h(t,z)=\int_Z h_0(y)p_t(z,\nu(y))
\]
Additionally, $p_t(z,\cdot)$ is the transition function of a strong Markov process whose paths are right continuous and have no other discontinuities than jumps.
\end{proposition}

\begin{proof}
The proof will be the same as in the case of \cite[Thm.\ 4.2]{ZunigaNetworks}, and is
given here for the convenience of the reader.
\newline

Using the correspondence between Feller semigroups and transition func-
tions, one sees from Lemma \ref{generalFellerSemigroup} that there is a uniformly stochastically continuous $C_0$-transition function $p_t(z,\nu)$ satisfying condition (L) of \cite[Thm.
2.10]{Taira2009} such that
\[
\exp\left(t\mathscr{L}\right)h_0(z)=
\int_Z h_0(y)p_t(z,\nu(y))
\]
for $h_0\in C(Z)$, cf.\ e.g.\ \cite[Thm.\ 2.15]{Taira2009}. Now, by using the correspondence between transition functions and Markov processes, there exists
a strong Markov process whose paths are right continuous and have no discontinuities other than jumps, see e.g.\
\cite[2.12]{Taira2009}.
\end{proof}

\subsection{Markov process on $X(K)$}

It will be assumed here that
\[
S=V(\omega)\subset V(f)
\]
holds true, where $\omega$ is a regular $\Gamma$-invariant differential form on $\Omega$, and $f\in K(X)$ is a rational function on the Mumford curve $X$ over $K$. As before, it will be assumed that all zeros of $\omega$ and all zeros and poles of $f$ are $K$-rational points.

\begin{thm}
There exists a probability measure $p_t(x,\cdot)$ with $t\ge0$, $x\in X(K)\setminus\bar{S}$, on the Borel $\sigma$-algebra of $X(K)\setminus\bar{S}$ such that the Cauchy problem 
for the heat equation
\[
\left(\frac{\partial}{\partial t}-\mathcal{H}_f\right)h(t,x)=0
\]
with $t\ge0$, $x\in X(K)\setminus\bar{S}$, having initial condition $h(0,x)=h_0(x)\in C(X(K)\setminus\bar{S})$ has a unique solution in $C^1\left((0,\infty),X(K)\setminus\bar{S}\right)$ of the form
\[
h(t,x)=\int_{X(K)\setminus\bar{S}} h_0(y)p_t(x,\absolute{\omega(y)})
\]
Additionally, $p_t(x,\cdot)$ is the transition function of a Markov process whose paths are right continuous and have no other discontinuities than jumps.
\end{thm}

This is Theorem 2 of the introduction.

\begin{proof}
Lemma \ref{summableWeight} together with Theorem \ref{spectrumDiffusionOperator} show that
$\mathcal{A}_f$ is a Hilbert-Schmidt integral operator, since $\mathcal{A}_f$ acts trivially on 
the wavelets supported in discs of $X(K)\setminus\bar{S}$ 
by  Lemma \ref{waveletEigenvalue}.
Also, the degree function $\deg_{\mathcal{H}_f}(x)$ is finite for $x\in X(K)\setminus\bar{S}$, cf.\ Lemma \ref{niceKernelFunction}.3.
Hence,
$\mathcal{H}_f$ is an instance of the general operator $\mathscr{L}$ from Section \ref{generalFeller}, and Proposition \ref{generalMarkovProcess} applies.
\end{proof}














\section{Kernel functions via Riemann theta functions}

\subsection{Riemann theta functions for Mumford curves}

The Riemann theta function for a Mumford curve $X$ is defined in \cite[VI.3]{GvP1980}. The presentation here follows the lines of \cite[\S 1]{Steen1989}.
Define
\[
G_\Gamma=\Hom(\Gamma,\mathds{C}_p^\times)
\]
and identify this group with $\left(\mathds{C}_p^\times\right)^g$ using a fixed set of generators $\gamma_1,\dots,\gamma_g\in \Gamma$. If $b\in\Gamma a$ for $a,b\in\Omega$, then the automorphic form $\theta(a,b;z)$ does not depend on $a\in\Omega$ \cite[II.2.3.3]{GvP1980}.
In this case, the automorphy factor will be written as
\[
c_\gamma:=c(a,\gamma a;\cdot)\in G_\Gamma
\]
for $\gamma\in\Gamma$.
The group
\[
\Lambda_\Gamma=\mathset{c_\gamma\mid\gamma\in\Gamma}\subset G_\Gamma
\]
is a free abelian group of rank $g$ and can thus be identified with $\mathds{Z}^g$ \cite[VI.1.3]{GvP1980}.
\newline

There is an isomorphism
\begin{align}\label{JacobianTorus}
\Jac(X)\cong G_\Gamma/\Lambda_\Gamma
\end{align}
induced by the correspondence
\[
D=\sum\limits_{i=1}^n\left([\bar{a}_i]-\left[\bar{b}_i\right]\right)
\mapsto \prod\limits_{i=1}^n
c(a_i,b_i;\cdot)
\]
which takes a divisor of degree zero on the Mumford curve $X$ to an automorphy factor in $G_\Gamma$ \cite[VI.3.9]{GvP1980}.
The isomorphism (\ref{JacobianTorus}) reveals the Jacobian variety $\Jac(X)$ 
as an analytic torus.



\bigskip
The Riemann theta function of a Mumford curve $X$ of genus $g$ introduced in \cite[VI.3]{GvP1980}
is now given via
\[
\xi_{c_\gamma}\colon G_\Gamma\to\mathds{C}_p^\times,\;c\mapsto
P(c_\gamma,c_\gamma)c(\gamma)
\]
as
\[
\theta_\Gamma(c)=\sum\limits_{c_\gamma\in\Lambda_\Gamma}\xi_{c_\gamma}(c)
\]The
where 
\[
P(c_\gamma,c_\gamma)=Q(c_\gamma,c_\gamma)^{\frac12}\in\mathds{C}_p^\times\quad \text{(any choice of square root)}
\]
with
\[
Q\colon\Lambda_\Gamma\times\Lambda_\Gamma
\to\mathds{C}_p^\times
\]
the multiplicative bilinear form given as
\[
Q(c_\alpha,c_\beta)=\langle\alpha,\beta\rangle
\]
and where
\[
\langle\cdot,\cdot\rangle\colon\Gamma^{\abelian}\times\Gamma^{\abelian}\to\mathds{C}_p^\times,\;(\alpha,\beta)\mapsto\frac{u_\alpha(z)}{u_\alpha(\beta z)}
\]
is the period pairing of $X$ with $\Gamma^{\abelian}=\Gamma/[\Gamma,\Gamma]$.
Notice that the values do not depend on the choice of $z\in\Omega$. Using the  map 
\[
t_\Gamma\colon\Omega\to G_\Gamma,\;x\mapsto c(x,x_0;\cdot)
\]
for a fixed base point $x_0\in\Omega$, obtain the induced canonical embedding, called \emph{Abel-Jacobi map}
\[
\bar{t}_\Gamma\colon X\to\Jac(X)
\]
which extends to divisors in  a canonical way.
\newline

Finally, the function of interest is given as
\[
\vartheta=\theta_\Gamma\circ t_\Gamma\colon \Omega\to\mathds{C}_p
\]
which is, according to the Riemann Vanishing Theorem \cite[Thm.\ VI.3.8]{GvP1980}, is a holomorphic function whose divisor is $\Gamma$-invariant, and, considered as a divisor  
in $\Div(X)$, has degree $g$.
\newline

It is helpful to formulate this theory in the language of \cite{Steen1989} or \cite{Teitelbaum1988}, since these articles contain explicit calculations for special kinds of Mumford curves. In particular, a more general form of the functional equation for $\vartheta$, and the $\vartheta(c;\cdot)$ to be defined below, than as stated in  \cite[VI.3.3 and VI.3.4]{GvP1980} will be given.

\begin{Lemma}\label{PairingFactor}
It holds true that
\[
P(c_{\alpha\beta},c_{\alpha\beta})=P(c_\alpha,c_\alpha)c_\alpha(\beta) P(c_\beta,c_\beta)
\]
for $\alpha,\beta\in \Gamma$.
\end{Lemma}

\begin{proof}
It holds true that
\begin{align*}
P(c_{\alpha\beta},c_{\alpha\beta})
&=P(c_\alpha,c_\alpha)P(c_\alpha,c_\beta)^2P(c_\beta,c_\beta)
\\
&=P(c_\alpha,c_\alpha)Q(c_\alpha,c_\beta)P(c_\beta,c_\beta)
\end{align*}
by the multiplicative bilinearity of $P(\cdot,\cdot)$. Now,
\begin{align*}
Q(c_\alpha,c_\beta)&=\frac{u_\alpha(z)}{u_\alpha(\beta z)}
=c(a,\beta a;\alpha)=c_\beta(\alpha)
\end{align*}
which, by symmetry of $Q(\cdot,\cdot)$ also equals $c_\alpha(\beta)$. This proves the assertion.
\end{proof}

\begin{Corollary}[Functional Equation 1]\label{FE1}
It holds true that
\[
\theta_\Gamma(cc_\alpha)=\xi_{c_\alpha}(c)^{-1}\theta_\Gamma(c)
\]
for $\alpha\in\Gamma$ and $c\in G_\Gamma$.
\end{Corollary}

\begin{proof}
It holds true that
\begin{align*}
\xi_{c_\alpha}(c)\xi_{c_\gamma}(cc_\alpha)
&=
P(c_\alpha,c_\alpha)c(\alpha)P(c_\gamma,c_\gamma)c(\gamma)c_\alpha(\gamma)
\\
&=P(c_{\alpha\gamma},c_{\alpha\gamma})c(\alpha\gamma)
=\xi_{c_{\alpha\gamma}}(c)
\end{align*}
according to Lemma \ref{PairingFactor}.
This proves the assertion.
\end{proof}

\begin{Lemma}\label{pullout}
It holds true that
\[
t_\Gamma(z)\cdot c_\alpha=t_\Gamma(\alpha^{-1}z)
\]
for $z\in\Omega$ and $\alpha\in\Gamma$.
\end{Lemma}

\begin{proof}
It holds true that
\begin{align*}
\left(t_\Gamma(z)\cdot c_\alpha\right)(\gamma)&=c(z,x_0;\gamma)c(\alpha^{-1}x_0,x_0;\gamma)
\\
t_\Gamma(\alpha^{-1}z)(\gamma)&=c(\alpha^{-1}z,x_0;\gamma)
\end{align*}
The first product is
\[
\prod\limits_{\delta\in\Gamma}
\frac{x-\delta z}{x-\delta x_0}\cdot\frac{x-\delta\alpha^{-1}x_0}{x-\delta x_0}\cdot\frac{\gamma x-\delta x_0}{\gamma x-\delta z}\cdot\frac{\gamma x-\delta x_0}{\gamma x-\delta\alpha^{-1} x_0}
\]
for $x\in\Omega$. Now, view this as four infinite products over $\delta\in\Gamma$, and substitute in the first and the third factor $\delta\mapsto\alpha^{-1}\delta$. This does not change its value, and one obtains
\[
\prod\limits_{\delta\in\Gamma}\frac{x-\delta\alpha^{-1}z}{x-\delta x_0}\cdot\frac{\gamma x-\delta x_0}{\gamma x-\delta\alpha^{-1}}=c(\alpha^{-1}z,x_0;\gamma)
\]
from which the assertion now follows.
\end{proof}

Let $c\in G_\Gamma$, and define
\[
\vartheta(c;z)=\theta_\Gamma(c\cdot t_\Gamma(z))
\]
for $z\in\Omega$.

\begin{Corollary}[Functional Equation 2]\label{FE2}
It holds true that
\begin{align*}
\vartheta(c;\alpha^{-1}z)
&=c(\alpha)^{-1}\xi_{c_\alpha}(t_\Gamma(z))^{-1}\vartheta(c;z)
\\
&=c(\alpha)^{-1}P(c_\alpha,c_\alpha)^{-1}c(z,x_0;\alpha)^{-1}\vartheta(c;z)
\end{align*}
for $z\in\Omega$, $\alpha\in\Gamma$, and $c\in G_\Gamma$.
\end{Corollary}

\begin{proof}
One calculates, using Lemma \ref{pullout} and Corollary \ref{FE1}, that
\begin{align*}
\vartheta(c;\alpha^{-1}z)
&=\theta_\Gamma(c\cdot t_\Gamma(\alpha^{-1}z))
=\theta_\Gamma(c\cdot t_\Gamma(z)c_\alpha)
\\
&=\xi_{c_\alpha}(c\cdot t_\Gamma(z))^{-1}\theta_\Gamma(c\cdot t_\Gamma(z))
\\
&=c(\alpha)^{-1}\xi_{c_\alpha}(t_\Gamma(z))^{-1}\vartheta(c;z)
\\
&=c(\alpha)^{-1}P(c_\alpha,c_\alpha)^{-1}c(z,x_0;\alpha)^{-1}\vartheta(c;z)
\end{align*}
which proves the assertion.
\end{proof}

An immediate consequence of the functional equation Corollary \ref{FE2} is that the divisor
\[
\divisor\left(\vartheta(c;z)\right)
\]
is $\Gamma$-invariant.
\newline

In order to compare with \cite{GvP1980}, 
observe the following:
\begin{Lemma}
It holds true that
\[
\vartheta(c;z)=
\sum\limits_{c_\gamma\in\Lambda_\Gamma}P(c_\gamma,c_\gamma)c(\gamma)u_\gamma(z)
\]
for $c\in G_\Gamma$ and $z\in\Omega$.
\end{Lemma}
 
\begin{proof}
It holds true that
\[
\vartheta(c;z)=\sum\limits_{c_\gamma\in\Lambda_\Gamma}\xi_{c_\gamma}(c\cdot t_\Gamma(z))
\]
and
\[
\xi_{c_\gamma}(c\cdot t_\Gamma(z))=P(c_\gamma,c_\gamma)\cdot c(\gamma)\cdot c(z,x_0;\gamma)
\]
where further
\begin{align}\label{cugamma}
c(z,x_0;\gamma)=\frac{u_\gamma(z)}{u_\gamma(x_0)}=u_\gamma(z)
\end{align}
because
\[
u_\gamma(x_0)=\theta(x_0,\gamma x_0;x_0)=1
\]
holds true. This implies the assertion.
\end{proof}

Now, Gerritzen and van der Put in \cite[VI]{GvP1980} formulate this with respect to a fixed basis $\gamma_1,\dots,\gamma_g$ of $\Gamma$. So, if 
\[
\gamma\equiv\gamma_1^{n_1}\cdots\gamma_g^{n_g}\mod[\Gamma,\Gamma]
\]
for $n_1,\dots,n_g\in\mathds{Z}$, then
\[
c(\gamma)=c_1^{n_1}\cdots c_g^{n_g}
\]
with 
\[
c_i=c(\gamma_i),\quad i=1,\dots,g
\]
giving rise to
\[
c(\gamma)\cdot u_\gamma(z)
=c_1^{n_1}u_1(z)^{n_1}\cdots c_g^{n_g}u_g(z)^{n_g}
=(c_1u_1(z))^{n_1}\cdots(c_gu_g(z))^{n_g}
\]
with
\[
u_i(z)=u_{\gamma_i}(z),\quad i=1,\dots,g
\]
in turn giving rise to the expression
\begin{align}\label{OurThetaFunction}
\vartheta(c;z)=\sum\limits_{n\in\mathds{Z}^g}P(n,n)(c_1u_1(z))^{n_1}\cdots(c_gu_g(z))^{n_g}
\end{align}
with
\[
P(n,n)=P(\gamma_1^{n_1}\cdots\gamma_g^{n_g},\gamma_1^{n_1}\cdots\gamma_g^{n_g})
\]
for $n=(n_1,\dots,n_g)\in\mathds{Z}^g$.
Notice that this action of $G_\Gamma$ is well-defined, because
\[
u_{\alpha\beta}=u_\alpha\cdot u_\beta
\]
for $\alpha,\beta\in\Gamma$, which implies that $u_\alpha=u_\beta$, if $\alpha\equiv\beta\mod[\Gamma,\Gamma]$, cf.\ \cite[II.2.3.5]{GvP1980}.
\newline

Using (\ref{cugamma}), the second functional equation takes the following form:
\begin{align}\label{FE3}
\vartheta(c;\alpha^{-1}z)
=c(\alpha)^{-1}P(c_\alpha,c_\alpha)^{-1}u_\alpha(z)^{-1}\vartheta(c;z)
\end{align}
which generalises the functional equation for the function named $f(c;z)$ in \cite[VI.3.4]{GvP1980}.

\begin{Lemma}
If $\vartheta(c;z)$ 
with $c\in G_\Gamma$
does not vanish everywhere on $\Omega$, then the divisor
\[
\divisor(\vartheta(c;z))\in\Div(X)
\]
is a well-defined divisor on $X$ and
has degree $g$.
\end{Lemma}

\begin{proof}
The divisor $\divisor(\vartheta(c;z))$ being $\Gamma$-invariant follows from functional equation 2 (Corollary \ref{FE2}). The statement about the degree is proven in
\cite[VI.3.4]{GvP1980}.
\end{proof}

Now, it is possible to define a $\Gamma$-invariant function
\[
f_\Gamma(z)=\frac{\vartheta(c;z)\cdot\vartheta(c'c'';z)}{\vartheta(cc';z)\cdot\vartheta(c'';z)}
\]
for $z\in\Omega$, where $c,c',c''\in G_\Gamma$ are chosen such that none of the occurring
theta functions vanish identically on $\Omega$. The kernel function $H_f$ for $f=f_\Gamma$ as a function on the Mumford curve $X(K)=\Omega(K)/\Gamma$ is to be used.

\begin{remark}
Equation (\ref{OurThetaFunction}) shows that the Riemann theta function $\vartheta(c;z)$ coincides with the function in \cite[VI.3.4]{GvP1980} written out there as $f(c;z)$. 
Notice also that their function $\vartheta$ is the function $\theta_\Gamma$ used here. Since the symbol $f$ denotes a meromorphic function on $X(K)$, we are not using their notation here.
\end{remark}

\subsection{Hyperelliptic Mumford curves}

In \cite[\S 3]{Steen1989}, G.\ van Steen studies the theta divisor of a $p$-adic hyperelliptic Mumford curve $X$, and finds that
if a character
\[
c\colon G_\Gamma\to \mathds{C}_p^\times
\]
has its image in $\mathset{\pm1}^g\subset K^\times$, then
the support of the divisor $\divisor(\vartheta(c;z))$ lies in the ramification locus of
the map
\[
\sigma\colon X\to\mathds{P}^1
\]
given by the hyperelliptic involution of $X$. This is implied by the statement following from and written after \cite[Lem.\ 3.1]{Steen1989}. 
\newline

In order to understand this, notice that in the hyperelliptic case, $\Gamma$ and the involution of $X$ generate a  discrete Whittaker group $W$ in  $\PGL_2(K)$ generated by involutions 
\[
s_0,s_1,\dots s_g
\]
where
\[
s_i=\gamma_is_0,\quad i=1,\dots,g
\]
and $s_0$ is an elliptic M\"obius transformation of order two, and where $\gamma_1,\dots,\gamma_g$ is a basis of $\Gamma$ such that
\[
s\gamma_is^{-1}=\gamma_i^{-1},\quad i=1,\dots,g
\] 
cf.\ \cite[Cor.\ 1.3 and Prop.\ 1.4]{Steen1983}. 
The ramification points of the map 
\[
\sigma\colon X\to\mathds{P}^1
\]
induced by the involution
are the $\Gamma$-orbits of the $2g+2$ fixed points 
of the transformations $s_0,\dots,s_g$. 
\newline

Using the notation
\[
c_{xy}=c(x,y;\cdot)
\]
for $x,y\in\Omega$ obtain the following:

\begin{Lemma}\label{SteenLemma}
It holds true that
\begin{align*}
c_{b_ia_0}&=-c_{a_ia_0},
\\
c_{a_ia_0}^2&=c_{b_ia_0}^2=c_{\gamma_i}
\\
c_{a_ia_0}(\gamma)&=\pm P(c_\gamma,c_\gamma)\quad(\gamma\in\Gamma)
\end{align*}
where $a_i,b_i$ are the fixed points of $s_i$ for $i=1,\dots,g$.
\end{Lemma}

\begin{proof}
The second line can be also seen 
using \cite[Prop.\ 1.3]{vanSteen1984}, which in the hyperelliptic case calculates this in order to show that  $c_{a_ia_0},c_{b_ia_0}\in\Jac(X)$ are both of order two.
The third line can also be extracted from the proof of \cite[Prop.\ 1.3]{vanSteen1984}, using that $P(c_\gamma,c_\gamma)=c_\gamma(\gamma)$. The first line follows from
\[
c_{b_ia_0}=c_{b_ia_i}c_{a_ia_0}
\]
shown in \cite[Page 35]{vanSteen1984}, and
\[
c_{b_ia_i}(\gamma_j)=(-1)^{\delta_{ij}}
\]
for all $i=1,\dots,g$, also calculated in \cite[Page 35]{vanSteen1984}.
\end{proof}

Since for $P(c_\gamma,c_\gamma)$ a square root is to  be chosen, the
following choice will be made:
\[
c_{a_ia_0}(\gamma)=P(c_\gamma,c_\gamma)
\]
for $\gamma\in\Gamma$.
Grouping the $2g+2$ ramification points in $X(K)$ as follows:
\[
\mathset{\bar{a}_0,\bar{b}_0,\dots,\bar{a}_g,\bar{b}_g}
\]
with $a_0,b_0,\dots,a_g,b_g\in\Omega(K)$,
and assuming that
\[
\divisor f_\Gamma(z)=\sum\limits_{i=1}^g [\bar{a}_i]
\]
with $\bar{x}=\Gamma x$ for $x\in\Omega$,
it follows from Lemma \ref{SteenLemma} that
\begin{align}\label{ElementaryDivisor}
\divisor f_\Gamma(c_{b_ia_0};z)
=\sum\limits_{j=1\atop j\neq i}^g
[\bar{a}_j]+\left[\bar{b}_i\right]
\end{align}
with $i=1,\dots,g$.
\newline

Let $I\subset\mathset{1,\dots,g}$ and $I^c=\mathset{1,\dots,g}\setminus I$, and set
\[
c_I=\prod\limits_{i\in I}c_{b_ia_0}
\]
in order to further manipulate the divisor of the function:

\begin{Lemma}
It holds true that
\[
\divisor f_\Gamma(c_I;z)=\sum\limits_{i\in I^c}[\bar{a}_i]+\sum\limits_{i\in I}\left[\bar{b}_i\right]
\]
for any $I\subset\mathset{1,\dots,g}$.
\end{Lemma}

\begin{proof}
This is an immediate consequence of (\ref{ElementaryDivisor}).
\end{proof}

In order to construct a meromorphic function on $X$, observe now that
\begin{align}\label{multiplyCharacters}
c_I\cdot c_J=c_{I\Delta J}
\end{align}
where 
\[
I\Delta J=(I\setminus J)\cup (J\setminus I)=(I\cup J)\setminus (I\cap J)
\]
is the symmetric difference of  $I,J\subset\mathset{1,\dots,g}$. Now, define
\[
f_{I,J,K}(z):=
\frac{f_\Gamma(c_I;z)f_\Gamma(c_Jc_K;z)}{f_\Gamma(c_Ic_J;z)f_\Gamma(c_K;z)}
\]
for $I,J,K\subset\mathset{1,\dots,g}$.

\begin{example}\label{standardFunction}
Let 
\begin{align*}
J_i=\mathset{i},
\quad
K_i=\mathset{i,i+1\mod g},\quad i=1,\dots,g,
\end{align*}
and set
\[
F(z)=\prod\limits_{i=1}^g f_{\emptyset,J_i,K_i}(z)
\]
Then one checks that
\[
\divisor(F(z))=\sum\limits_{withi=1}^g2\left([\bar{a}_i]-\left[\bar{b}_i\right]\right)
\]
which gives an explicit example of a function vanishing at all ramification points of the hyperelliptic Mumford curve $X$.
\end{example}

In \cite[\S6]{Gerritzen1985}, L.\ Gerritzen gives an explicit uniformisation of a hyperelliptic Mumford curve $X$ in the case $\Char(K)\neq 2$ as
\begin{align*}
x(z)&=1+4\left(\frac{1}{z-1}+\frac{1}{(z-1)^2}\right)
\\
&+4\sum\limits_{\gamma\in\Gamma\setminus 1}
\left(
\frac{1}{\gamma(z)-1}-\frac{1}{\gamma(1)-1}
+\frac{1}{(\gamma(z)-1)^2}-\frac{1}{(\gamma(1)-1)^2}
\right)
\end{align*}
where it is assumed that $\infty\notin\Omega$, and the fixed points of $s_0$ are $\pm 1$.
The fixed points of the other $s_1,\dots,s_g$ are given in terms of the fixed points of the fixed points of the transformations $\gamma_1,\dots,\gamma_g$.
From this uniformisation, it is possible to calculate the measure $\absolute{\omega_0(z)}$ for the differential $1$-form
\[
\omega_0(z)=\frac{dx(z)}{2y(z)}=\frac{x'(z)\,dz}{\sqrt{\prod\limits_{i=1}^{2g+1}(x(z)-e_i)}}
\]
where $e_0,\dots,e_{2g+1}\in\mathds{P}^1(K)$
are the branch points of the covering $\sigma\colon X\to\mathds{P}^1$ with
\[
e_0=x(1)=\infty,\;
e_1=x(-1)=0,\;
e_{2i}=x(a_i),\;
e_{2i+1}=x(b_i)
\]
for $i=1,\dots,g$. The equation corresponding to the covering map is
\[
y^2=x\cdot\prod\limits_{i=2}^{2g+1}(x-e_i)
\]
according to \cite[\S6]{Gerritzen1985}.
Observe that
\[
x'(z)=-4\left(\frac{1}{(z-1)^2}+\frac{1}{2(z-1)^3}\right)
-4\sum\limits_{\gamma\in\Gamma\setminus 1}
\left(
\frac{\gamma'(z)}{(\gamma(z)-1)^2}+
\frac{\gamma'(z)}{2(\gamma(z)-1)^3}
\right)
\]
and now obtain some absolute values. W.l.o.g.\ assume that the good fundamental domain $F$ is obtained by punching out some holes in $O_K$, and that $\pm1\in F$, but $0\notin F$. In fact, it is assumed that the fixed points of $\gamma_1$ are $0,\infty$. Also, because of $s\gamma_is^{-1}=\gamma_i$ for $i=1,\dots,g$, it holds true that the fixed points of each $\gamma_i$ are inverse to each other, but this is not needed in the following.

\begin{Lemma}\label{near1}
It holds true that
\begin{align*}
\absolute{x(z)}&=\frac{1}{\absolute{z-1}^2},\quad
\absolute{x'(z)}=\frac{1}{\absolute{z-1}^3}
\end{align*}
for $\absolute{z-1}<<1$.
\end{Lemma}

\begin{proof}
Assume that $\absolute{z-1}<C<<1$.
Since $1\in F$, it follows that $\gamma(z)$ is in one of the holes of $F$ for $\gamma\neq 1$. Hence, 
$\absolute{\gamma(z)-1}\ge C$. It immediately follows that
\[
\absolute{x(z)}=\absolute{1+4\left(\frac{1}{z-1}+\frac{1}{(z-1)^2}\right)}=\frac{1}{\absolute{z-1}^2}
\]
which is the first asserted equality.

\smallskip
Now,  $\gamma\in\Gamma\setminus 1$ can be represented by a matrix
\[
\begin{pmatrix}
a&b\\1&d
\end{pmatrix}
\]
with $a,b,d\in K$.
It holds true that
\[
\absolute{\gamma'(z)}=\frac{\absolute{\det(\gamma)}}{\absolute{z+d}^2}
\]
Since, by hyperbolicity, 
\[
\frac{\absolute{a+d}^2}{\absolute{\det(\gamma)}}>1
\]
assume w.l.o.g.\ that $\absolute{a}\le\absolute{d}$, since otherwise replace $\gamma$ with $\gamma^{-1}$ which is also hyperbolic.
Then 
\[
\absolute{ad-b}=\absolute{\det(\gamma)}<\absolute{a+d}^2\le
\absolute{d}^2
\]
and
\[
\absolute{\gamma'(z)}<
\frac{\absolute{d}^2}{\absolute{z+d}^2}
\]
If $\absolute{d}\neq1$, this last expression  is $\le1$.
In all these cases, it follows that
\[
\absolute{x'(z)}=\absolute{\frac{1}{(z-1)^2}+\frac{1}{(z-1)^3}}
=\frac{1}{\absolute{z-1}^3}
\]
as asserted.

\smallskip
So, now assume that $\absolute{d}=1$. 
If $\absolute{1+d}<\absolute{z-1}$, then
\[
\absolute{\gamma'(z)}<\frac{1}{\absolute{z+d}^2}=\frac{1}{\absolute{z-1+1+d}^2}=\frac{1}{\absolute{z-1}^2}
\]
which implies again the assertion for $\absolute{x'(z)}$.
If
$\absolute{z-1}<\absolute{1+d}\le 1$, then
\[
1\le\absolute{\gamma'(z)}=\frac{1}{\absolute{1+d}^2}<\frac{1}{\absolute{z-1}^2}
\]
again implying the assertion for $\absolute{x'(z)}$.
Now, if 
\begin{align*}\label{condition_d}
\absolute{1+d}=\absolute{z-1}<C
\end{align*}
then
$\absolute{\gamma'(z)}$
 could be arbitrarily large as $C\to0$.
 Assume that there is a sequence of $\gamma_n\in\Gamma$ represented in the above way, such that their $d$-entries tend to $-1$. Then there is a limit element $\alpha\in\PGL_2(K)$ of this sequence. Hence, only finitely many of these $\gamma_n$ are in $\Gamma$.
And this now implies the assertion for $\absolute{x'(z)}$ for $z\in F$ sufficiently close to $1$. The assertions are now proven.
\end{proof}

\begin{Corollary}\label{MotherDifferential}
The regular differential $1$-form $\omega_0(z)$ has a  zero in $z=1$ of order $2g-2$, and no other zeros.
\end{Corollary}

\begin{proof}
According to Lemma \ref{near1}, this is implied by
\[
\absolute{\omega_0(z)}=
\frac{\absolute{x'(z)}\absolute{dz}}{{C\cdot \absolute{x(z)}^g}}
=
\frac{\absolute{x'(z)}\absolute{dz}}{C\cdot\absolute{x(g)}^{\frac{2g+1}{2}}}
=\absolute{z-1}^{2g-2}\absolute{dz}
\]
for some $C>0$, if $\absolute{z-1}<<1$, because then 
\[
\absolute{x(z)-e_i}=\absolute{x(z)}
\]
holds true.
Since the divisor of any meromorphic differential $1$-form on a projective algebraic curve  defined over an  algebraically closed field has degree $2g-2$ by Riemann-Roch, the assertion now follows.
\end{proof}

From the theory of hyperelliptic curves, it is known that the space of $\Gamma$-invariant holomorphic differential $1$-forms
has a basis as follows:
\[
\omega_0,x\,\omega_0,\cdots,x^{g-1}\,\omega_0
\]
which allows to prescribe the divisor of $\omega\in H^0(X,\Omega_{X/K})$ which has degree $2g-2$. So, in order to meet the requirement $S\subseteq V(F(z))$ of Lemma \ref{HilbertSchmidtCondition}, define
\[
\omega:=\prod\limits_{i=1}^g(x-e_{2i})^{m_i}\,\omega_0
\]
with natural $m_1,\dots,m_g\ge0$ such that
\[
\sum\limits_{i=1}^{g}m_i=2g-2
\]
holds true.

\begin{Corollary}\label{hyperellipticDifferential}
It holds true that
\[
\divisor(\omega)=\sum\limits_{i=1}^gm_i\left[\bar{a}_i\right]
\]
with $\omega$ viewed as a regular differential $1$-form on $X(K)$.
\end{Corollary}

\begin{proof}
First, observe that
\[
\divisor(x)=\left[\,\overline{-1}\,\right]-\left[\,\bar{1}\,\right]
\]
viewed as an element of $\Div(X)$. The mulitplicities follow from the fact that the $\Gamma$-orbits of $1$ and $-1$  are ramification points
of the $2$-sheeted covering $X\to\mathds{P}^1$ induced by the $\Gamma$-invariant function $x$. 
In particular, $1\in F$ is a simple pole of $x$.
Now,
using Corollary \ref{MotherDifferential}, check that
\begin{align*}
\divisor(\omega)&
=\divisor\left(\prod\limits_{i=1}^g(x-e_{2i})^{m_i}\right)+\divisor(\omega_0)
\\
&=
\sum\limits_{i=1}^gm_i\left[\bar{a}_i\right]
-(2g-2)\left[\,\bar{1}\,\right]+\divisor(\omega_0)
\\
&=\sum\limits_{i=1}^gm_i\left[\bar{a}_i\right]
\end{align*}
as asserted.
\end{proof}

The results obtained here mean that in the case of a hyperelliptic Mumford curve, the kernel function of $\mathcal{H}_F$ can be given in terms of the ramification points $a_0,b_0,\dots,a_g,b_g\in\Omega$.

\subsection{Hearing the genus of a Mumford curve}

In the following, assume that $X$ is a Mumford curve defined over $K$ of genus $g\ge 2$, $F\subset K$ a good fundamental domain for $X$, and $\gamma_1,\dots,\gamma_g$ corresponding generators of the Schottky group $\Gamma$, whose fixed points are all $K$-rational, and such that the fixed points of $\gamma_1$ are $0,\infty$. Again, assume that
\[
\divisor(F(z))=\sum\limits_{a\in V(F(z))}m_a[a]-\sum\limits_{b\in P(F(z))}m_b[b],\quad
\divisor(\omega)=
\sum\limits_{a\in V(\omega)}n_a[a]
\]
with
\[
\sum\limits_{a\in V(F(z))}m_a-\sum\limits_{b\in P(F(z))}m_b=0,\quad
\sum\limits_{a\in V(\omega)}n_a=2g-2
\]
where $P(F(z))$ is the set of poles of $F(z)$.

\begin{Lemma}\label{VertexCondition1}
Let $x\in F\setminus S$ be close to $a\in V(\omega)$, say at distance $p^{-fr_x}$. Assume that $y\in \dot{B}_{r_x+1}(a)=B_{r_x+1}(a)\setminus \mathset{a}$. 
Then
\[
\int_{\dot{B}_{r_x+1}(a)}
H(x,y)\absolute{\omega(y)}=C
p^{-kfm_a}p^{-fn_ad_x}
\]
for $C>0$ independent of $x$, and $r_x>>0$.
\end{Lemma}

\begin{proof}
It holds true that
\[
\absolute{x-y}=p^{-fr_x}
\]
for $y\in \dot(B)_{r_x+1}(a)$.
Since $0\notin\Omega$, it does happen that $x-y$ falls into a hole of $F$, if $r_x$ becomes too large. In this case, replace $x-y$ with
$\mu_1^{n}(x-y)$ for suitable $n\in\mathds{Z}$, where $\mu_1\in \pi O_K$ is the multiplier of $\gamma_1$. So, since
\[
\absolute{x-y-a}=\absolute{x-y}
\]
it follows, as Lemma$x$ tends to $a\in V(\omega)$, that $H(x,y)$ takes only the following values:
\[
Cp^{-fm_ak},\quad k=0,\dots,v(\mu_1)-1
\]
for $C>0$ independent of $k$.
Since 
\[
\absolute{\omega(y)}
=\absolute{y-a}^{n_a}\absolute{dy}
\]
it now follows that
\[
\int_{\dot(B)_{r_x+1}(a)}
H(x,y)\absolute{\omega(y)}
=Cp^{-fm_ak}\sum\limits_{\ell=r_x+1}^\infty p^{-fn_a\ell}(1-p^{-f})
=Cp^{-fm_a k}p^{-fn_ar_x}
\]
as asserted.
\end{proof}

Denote with
\[
v(a,b,c)\in\mathscr{T}_K
\]
the unique vertex in the Bruhat-Tits tree $\mathscr{T}_K$ determined by $a,b,c\in\mathds{P}^1(K)$.

\begin{Lemma}\label{VertexCondition2}
Let $x\in F\setminus S$ be close to $a\in V(\omega)$ at distance $p^{-fr_x}$, and $v$ a vertex of $T_{F\setminus S}$ between $v_x=v(a,x,\infty)$ and the vertex $v_0=v(0,1,\infty)$ of $T_{F\setminus S}$, but $v\neq v_x$. Then
\[
\int_{U(v)}H(x,y)\absolute{\omega(y)}
\]
is independent of $x$ for $r_x>>0$.
\end{Lemma}

\begin{proof}
It holds true that
\[
\absolute{x-y}=p^{-fd(v,v_0)}
\]
where $d(v,w)$ is the distance in the Bruhat-Tits tree $\mathscr{T}_K$ induced by the Haar measure, i.e.\ each edge has length one.
It immediately follows that
$H(x,y)$ for $y\in U(v)$ is independent of $x$ 
under the given conditions.
\end{proof}

\begin{thm}\label{recoverGenus}
Let $X$ be a Mumford curve of genus $g$ defined over a non-archimedean local field $K$, and let $F\in K(X)$, and $\omega$ be a regular differential $1$-form on $X(K)$.
Assume that  $\divisor F(z),\divisor(\omega)\in \Div(X(K))$. 
Then it is possible to recover the genus $g(X)$ from the spectrum of $\mathcal{H}_F$. 
\end{thm}

This is Theorem 3 of the introduction.

\begin{proof}
Notice that $V(\omega)\subseteq V(F(z))$ is assumed. Assume also that $x\in F\setminus S$ is close to some $a\in V(\omega)$.
If $y\in U(v)$ with
$v\in T_{F\setminus S}$ a vertex not on the half-line beginning in $v_0=v(0,1,\infty)$ and ending in $a$, then
\[
\absolute{x-y}=\absolute{x-z}
\]
for $z\in U(w)$
with 
\[
w=v\wedge a
\]
where $\wedge$ is the join in the tree $T_S$ of ${F\setminus S}$ with root $v_0$.
By Lemmas \ref{VertexCondition1} and \ref{VertexCondition2}, it follows that
the set of negative degree eigenvalues of $\mathcal{H}_F$ has $n$ limit points, where $n=\absolute{S}$, where $S=V(\omega)$. Namely, the tree of $F\setminus S$ has precisely $n$ ends, and there is a convergence of the degree function 
\[
\deg_{\mathcal{H}_f}(x)\to \deg_{\mathcal{H}_F}(a)
\]
as $x\in F\setminus S$ tends to $a\in S$, assuming that $S$ is a subset of the good fundamental domain  $F$ for $X$, and the convergence can be made explicit as follows:
Assume a sequence $x_k\in F\setminus S$ given with
\[
\absolute{x_k-a}=p^{-fk}
\]
for natural $k>>0$.
Then
it holds true that
\[
\deg_{\mathcal{H}_f}(x_{k+1})=\deg_{\mathcal{H}_f}(x_k)+C\int\limits_{B_{k+1}(a)}\absolute{y-a}^{m_a}\absolute{dy}
\]
with suitable $C>0$ independent of $k$, for $k$ sufficiently large.
This means that from the rate of convergence of the degree function to  a limit point $a\in S$, given via
\[
C\cdot p^{-f(k+1)n_a}
\]
for $k>>0$,
the value of the natural number $n_a$ can be estimated.  Doing this for all limit points of $\deg_{\mathcal{H}_f}(x)$, it now follows from
\[
\deg(\divisor(\omega))=2g-2
\]
that the genus $g$ can be extracted.
\end{proof}

\begin{remark}
Theorem \ref{recoverGenus} means that, using Example \ref{standardFunction} and Corollary \ref{hyperellipticDifferential}
in the hyperelliptic case, it is now in principle possible to  explicitly construct an operator $\mathcal{H}_F$, calculate its negative degree eigenvalues, and then
recover the genus from that part of the spectrum. Since it has $\absolute{V(\omega)}$ limit points, and exponential convergence to these, an explicit genus recovery method is now (in principle) feasible in this case.
Another thing which the proof of that theorem  reveals, is that it is in principle possible to extract the absolute value of the multiplier $\mu_1$ of generator $\gamma_1$.
\end{remark}
\section*{Acknowledgements}
Andrew Bradley, \'Angel Mor\'an Ledezma, Martin M\"oller and David Weisbart are warmly thanked for fruitful discussions. Marius van der Put is thanked for sharing valuable insights about hyperelliptic Mumford curves. Frank Herrlich is thanked for introducing  the author to Mumford curves in the first place when he was his undergraduate student. This work is partially supported by the Deutsche Forschungsgemeinschaft under project number 469999674.

\bibliographystyle{plain}
\bibliography{biblio}

\end{document}